\documentclass[10pt]{article}

\usepackage{amsmath}
\usepackage{array}
\usepackage{amsfonts}
\usepackage{amssymb}
\usepackage{amsthm}
\usepackage[dvips]{graphics}
\usepackage{epsfig,graphics,graphicx}
\usepackage{stmaryrd}

\usepackage{tikz}
\usepackage{wrapfig} 
\usepackage[all]{xy}

\usepackage[babel=true]{csquotes}

\usepackage{mathtools}
\usepackage[numbers]{natbib}

\DeclarePairedDelimiter\abs{\lvert}{\rvert}%

  \theoremstyle{plain}
\newtheorem{theorem}{Theorem}[section]

\newtheorem{corollary}[theorem]{Corollary}
\newtheorem{lemma}[theorem]{Lemma} 
  \theoremstyle{remark}
\newtheorem{remark}[theorem]{Remark}
  \theoremstyle{definition}
\newtheorem{definition}[theorem]{Definition}

\newcommand{\ie}{\textit{i.e. }}

\newcommand{\NN}{\mathbb{N}}
\newcommand{\QQ}{\mathbb{Q}}
\newcommand{\ZZ}{\mathbb{Z}}
\newcommand{\RR}{\mathbb{R}}
\renewcommand{\SS}{\mathbb{S}}

\title{Polyak type equations for virtual arrow diagram invariants in the annulus}
\author{Arnaud Mortier \\ \itshape mortier@math.ups-tlse.fr}
\date{\today}

\begin{document}

\maketitle

\begin{abstract}
\footnotesize

We describe the space of arrow diagram formulas (defined in \citep{PolyakViro}) for virtual knot diagrams in the annulus $\RR\times\SS^1$ as the kernel of a linear map, inspired from a conjecture due to M. Polyak. As a main application, we slightly improve Grishanov-Vassiliev's theorem for planar chain invariants (\citep{Grishanov}).

\end{abstract}

\tableofcontents

\section{Introduction}

Gauss diagrams were introduced in knot theory for the purpose of extracting new combinatorial data from the widely studied knot diagrams. On one hand it gave rise to a generalization of knot theory, known as \textit{virtual knot theory} \citep{KauffmanVKT}. On another hand, it allowed a new point of view on Vassiliev's finite type invariants (see \citep{PolyakViro}, \citep{Goussarov}, \citep{Fiedler}). Several approaches have been used in order to define finite type invariants for virtual knots. Vassiliev-Kauffman's invariants \citep{KauffmanVKT} are directly inspired from the axiomatic definition of Vassiliev invariants given by J.Birman and X.-S.Lin \citep{BirmanLin}, while the approach of M.Goussarov, M.Polyak and O.Viro (GPV, \citep{GPV}) is inspired from the representation of Vassiliev invariants due to Goussarov \citep{Goussarov}. 

Another direction of investigation is the approach of T.Fiedler, who decorates Gauss diagrams with homological information when the knot diagrams live in a surface that is more complicated than the sphere or the plane. 

Here we focus on homogeneous GPV's invariants for virtual knot diagrams in the annulus.
\begin{itemize}
\item The annulus, because it has an abelian fundamental group. This property allows one to prove that Fiedler's decorated Gauss diagrams encode the knot diagrams \textit{faithfully} -- \ie with no loss of information \citep{MortierGaussDiagrams}.
\item Homogeneous GPV invariants, because as we will show it is the good framework to consider a conjecture of M.Polyak, who predicts the existence of a linear map whose kernel consists of Gauss diagram invariants.
\end{itemize}

Every result in this paper can be actually extended to the case of an arbitrary surface replacing the annulus (except for Theorem~\ref{thm:main} where the surface needs to be orientable) but it requires more complicated combinatorial tools. It will be done in a forthcoming paper.

\subsection*{Acknowledgements}

The author thanks Thomas Fiedler for introducing him to the subject of Gauss diagram invariants, and for useful remarks on the presentation. He also acknowledges useful corrections from Victoria Lebed, and thanks the referee for careful reading and lucid remarks.

\section{Algebraic structures in Gauss diagram invariants theory}

\textit{Warning}. Though every Gauss or arrow diagram in this article comes with homological markings due to the solid torus framework, we will often refer to works where this is not the case, since many notions do not depend on this. Though it is not always explicitly mentioned, everything depends on the value of a fixed integer $K$ which is the global marking of every diagram (see section~\ref{sec:GDspaces} below).

\subsection{Gauss diagram spaces}\label{sec:GDspaces}
Following T. Fiedler (\citep{Fiedler}, \citep{Fiedlerbraids}) we define a \textit{(decorated) Gauss diagram (of degree n)} as an oriented circle marked with an integer, and $n$ oriented chords (the \textit{arrows}, which are abstract, \ie only the endpoints matter), each one equipped with a sign (also \textit{writhe}) and an integer (its \textit{marking}), up to oriented homeomorphisms of the circle. It is to be understood that the $2n$ endpoints of the arrows are distinct. It is proved in \citep{MortierGaussDiagrams} that such Gauss diagrams are in $1$-$1$ correspondence with virtual knot diagrams in the annulus, up to \textit{usual and virtual} Reidemeister moves. We denote by $\mathcal{G}_n$ (resp. $\mathcal{G}_{\leq n}$) the $\QQ$-vector space freely generated by Gauss diagrams of degree $n$ (resp. ${\leq n}$), and set $\mathcal{G}= \varinjlim\mathcal{G}_{\leq n}$.

To the well-known Reidemeister moves for knot diagrams correspond $\mathrm{R}$-moves for Gauss diagrams (see Fig.\ref{pic:Rmoves} -- as usual, the unseen parts must be the same for all of the diagrams that belong to a given equation.). Beware that these moves depend on the homology class $K$ of the considered knot diagrams.

\begin{figure}[h!]
\centering 
\psfig{file=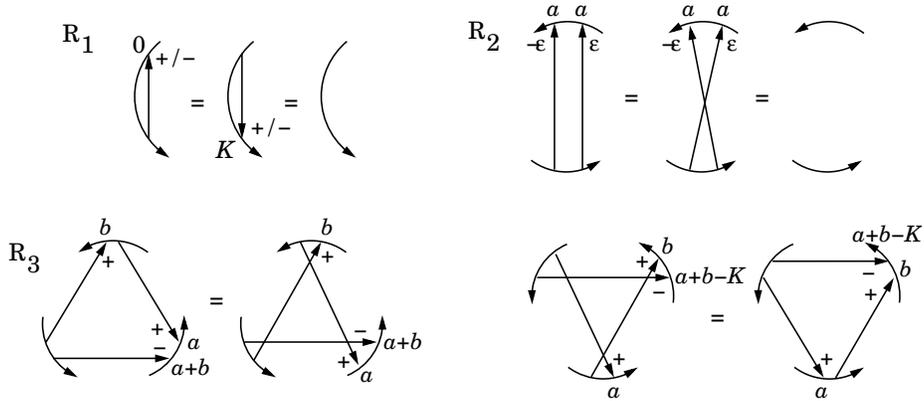,scale=0.62}
\caption{$\mathrm{R}$-moves for decorated Gauss diagrams with global marking $K$}\label{pic:Rmoves}
\end{figure}

We prove the following:

\begin{theorem}\label{GaussComplete}
The equivalence class of a Gauss diagram associated with a knot diagram, modulo the $\mathrm{R}$-relations of Fig.\ref{pic:Rmoves}, is a complete invariant for virtual knots with homology class $K$.
\end{theorem}

There is a linear isomorphism $I: \mathcal{G}_{\leq n}\rightarrow\mathcal{G}_{\leq n}$ that associates to a Gauss diagram the formal sum of its subdiagrams (see \citep{GPV}). A \textit{Gauss diagram formula} is a knot invariant of the form
\begin{eqnarray}\label{eq1}
\mathtt{G} \mapsto w(G,I(\mathtt{G})),
\end{eqnarray}
where $G\in \mathcal{G}$, $\mathtt{G}$ is the Gauss diagram associated with a knot projection and $w$ is an orthogonal scalar product with respect to the basis of $\mathcal{G}$ given by Gauss diagrams. Since this theory was born, mainly two scalar products have been used, namely:
\begin{itemize}
\item The \textit{orthonormal} scalar product, which we shall denote by $(,)$. It is used notably in \citep{GPV} and \citep{Fiedler}.
\item Its normalized version $ \left\langle , \right\rangle $, defined by
\begin{eqnarray}\label{eq2}
 \left\langle G,G^\prime  \right\rangle :=\abs{\operatorname{Aut}(G)}
\cdot(G,G^\prime),
\end{eqnarray}
where $\operatorname{Aut}(G)$ is the set of symmetries of $G$, \ie rotations that keep it unchanged (it is a subgroup of $\ZZ / 2n$).
\end{itemize}
Roughly speaking, $ \left\langle , \right\rangle $ counts parametrized configurations of arrows, while $(,)$ counts unordered sets of arrows. Notice that $ \left\langle , \right\rangle $ is still symmetric (hence a scalar product). Obviously, the two definitions coincide when one deals with long knots (and thus \textit{based} Gauss diagrams).

The second version was already defined in \citep{PolyakViro} (though their Theorem $2$ is stated in terms of $(,)$), but it was O.P.\"{O}stlund who first formally stated that $ \left\langle , \right\rangle $ is more convenient to get nice properties when dealing with Gauss diagrams with symmetries (\citep{Ostlund}, sections $2.2$ and $2.4$). The results that we present here confirm this fact.

The pairing $ \left\langle , \right\rangle $ is also used in \citep{Grishanov}, and implicitly in \citep{Vassiliev}. 

\subsection{Arrow diagram spaces}\label{sec:arrows}
Take a Gauss diagram $G$ and forget the signs associated with the arrows. We call what remains an \textit{arrow diagram} (see \citep{P1}; beware that the terminology in \citep{GPV} is different: arrows in arrow diagrams are signed). Arrow diagram spaces $\mathcal{A}_{n}$, $\mathcal{A}_{\leq n}$ and $\mathcal{A}$, and the pairings $(,)$, $ \left\langle , \right\rangle $ are defined similarly to the previous section.

The raison d'\^{e}tre of this notion lies in the following map: take an arrow diagram $A\in \mathcal{A}_n$ and number its arrows from $1$ to $n$. Then any map \mbox{$\sigma:\left\lbrace 1,\ldots,n \right\rbrace \rightarrow\left\lbrace \pm 1\right\rbrace$} defines a Gauss diagram $A^\sigma$. Let $\operatorname{sign}(\sigma)$ be the product of all the $\sigma(i)$'s. We put: 
\begin{eqnarray}\label{eq3}
S(A)\stackrel{\mathrm{def}}{=}\sum_{\sigma\in \left\lbrace \pm 1\right\rbrace^n}\operatorname{sign}(\sigma)\cdot A^\sigma.
\end{eqnarray}
$S$ extends linearly into a map $\mathcal{A}_n\rightarrow\mathcal{G}_n$. A Gauss diagram formula that lies in the image of this map is called an \textit{arrow diagram formula}. A lot of the explicit formulas that have been found so far are actually arrow diagram formulas -- as well in the framework of knots in $\SS^3$.

Considering this map is relevant only in the context of the $ \left\langle , \right\rangle $ pairing (\ref{eq2}). Indeed, one may define (as most authors do) brackets $(\!(A,G)\!)$ and \mbox{$ \left\langle \! \left\langle A,G \right\rangle \! \right\rangle $}, with $A\in \mathcal{A}$ and $G\in \mathcal{G}$ in the following way: for every subdiagram (\ie unordered set of arrows) of $G$ that becomes $A$ after one forgets its signs, form the product of these signs. Sum up all these products, and call the result $(\!(A,G)\!)$. On the other hand, put $ \left\langle \! \left\langle A,G \right\rangle \! \right\rangle :=\abs{\operatorname{Aut}(A)}
\cdot(\!(A,G)\!)$. Then, of the naturally expected relations
$$\begin{array}{cccc}
(\!(A,G)\!) & \stackrel{?}{=} & (S(A),I(G)) &\text{ and}\\
 \left\langle \! \left\langle A,G \right\rangle \! \right\rangle  & \stackrel{?}{=} &  \left\langle S(A),I(G) \right\rangle , &
\end{array}$$
only the second one holds true, while the first one needs the assumption that $A$ has no symmetries (see Lemma \ref{lem:brackets}).

A special interest arises in arrow diagram formulas in the case of virtual knot theory, as we shall see in the next subsection.

\subsection{Virtual knot invariants}

Virtual knot theory arises as the natural \enquote{completion} of classical knot theory with respect to Gauss diagrams. Indeed, while a knot diagram may be \textit{represented} by a Gauss diagram (with corresponding Reidemeister moves on Gauss diagrams), a virtual knot diagram actually \textit{is} a Gauss diagram. New (\enquote{virtual}) crossings are used as an artefact to draw planar representations of them, and the additional \textit{virtual Reidemeister moves} are precisely those planar moves that do not affect the underlying Gauss diagram (see \citep{KauffmanVKT}).

\subsubsection{Classical \textit{vs} virtual invariants}
One should be extremely cautious about the fact that the so-called \enquote{real} (or classical) Reidemeister moves for Gauss diagrams may not always be actually performed: for instance, two arrows may be added by Reidemeister II in the \textit{real} settings only if the corresponding arcs of the knot diagram face each other -- which seems not easy to check on the Gauss diagram.

As a consequence, the framework introduced previously seems mostly comfortable to look for virtual knot invariants.

A natural related question is whether a given Gauss diagram formula for classical knots always defines an invariant for virtual knots by the same \mbox{equation (\ref{eq1}).} The answer is negative, the simplest example is the formula for the invariant $v_3$ given by  \citep{PolyakViro} (Theorem 2), which we reproduce with an example of non invariance on Fig.\ref{pic:v3}.

\begin{figure}[h!]
\centering 
\psfig{file=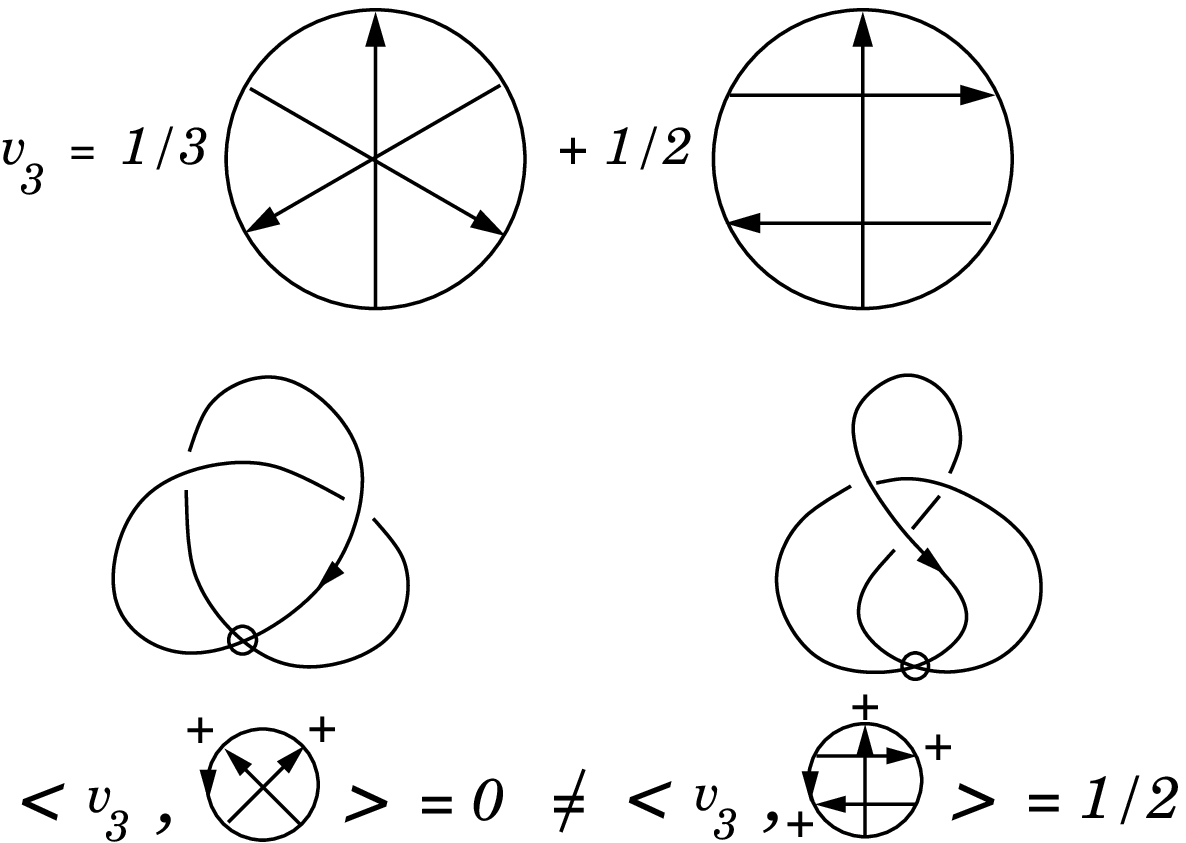,scale=0.7}
\caption{Polyak-Viro's formula for $v_3$ is not a virtual invariant}\label{pic:v3}
\end{figure}

\subsubsection{Homogenous virtual invariants}\label{sec:homogeneous}
\begin{definition}
For each $n\in \NN$, there is an orthogonal projection $\pi_n : \mathcal{G}\rightarrow \mathcal{G}_n$ with respect to the scalar product $(,)$. For $G\in \mathcal{G}$, there is some integer $n$ such that $G\in \mathcal{G}_{\leq n}\setminus \mathcal{G}_{\leq n-1}$. The \textit{principal part} of $G$ is defined by $\pi_n (G)$.
$G$ is called \textit{homogeneous} if it is equal to its principal part.
\end{definition}

\begin{lemma}\label{PrincipalPart}
Let $G\in \mathcal{G}$ be a Gauss diagram formula for \textit{virtual} knots. Then its principal part lies in the image of the map $S$ defined by (\ref{eq3}), \ie can be represented by a (homogeneous) arrow polynomial.
\end{lemma}
\begin{corollary}
Any \textit{homogeneous} Gauss diagram formula for virtual knots is an arrow diagram formula.
\end{corollary}

The above result in the context of knot theory in the sphere is contained in the lines \citep[section $3.1$]{GPV}, and the proof in our context contains no new ideas. What is new is that the converse is also true, in some sense:

\begin{theorem}\label{thm:ArrowHomogenous}
Let $\mathcal{IA}_{\leq n}$ be the space of arrow diagram formulas for virtual knots of degree no greater than $n$. Then:
$$\mathcal{IA}_{\leq n}=\bigoplus_{k\leq n} \left( \mathcal{IA}_{\leq n}\cap \mathcal{A}_{k}\right).$$
\end{theorem}

\subsection{The Polyak algebra}

A Gauss sum $G\in \mathcal{G}$ defines a virtual knot invariant if and only if the function $ \left\langle G,I(.) \right\rangle $ is well defined on the quotient of $\mathcal{G}$ by Reidemeister moves on Gauss diagrams. Hence it is interesting to understand the image of that subspace under the map $I$ with a simple family of generators. This is the idea that led the construction of the Polyak algebra (\citep{P1,GPV}) in the classical case. We adapt this construction and define $\mathcal{P}$ as the quotient of $\mathcal{G}$ by the set of relations shown in Fig.\ref{pic:Pmoves}, which we call $\mathrm{P}_1$, $\mathrm{P}_2$ and $\mathrm{P}_3$ (also $8T$) relations for Gauss diagrams.

\begin{figure}[h!]
\centering 
\psfig{file=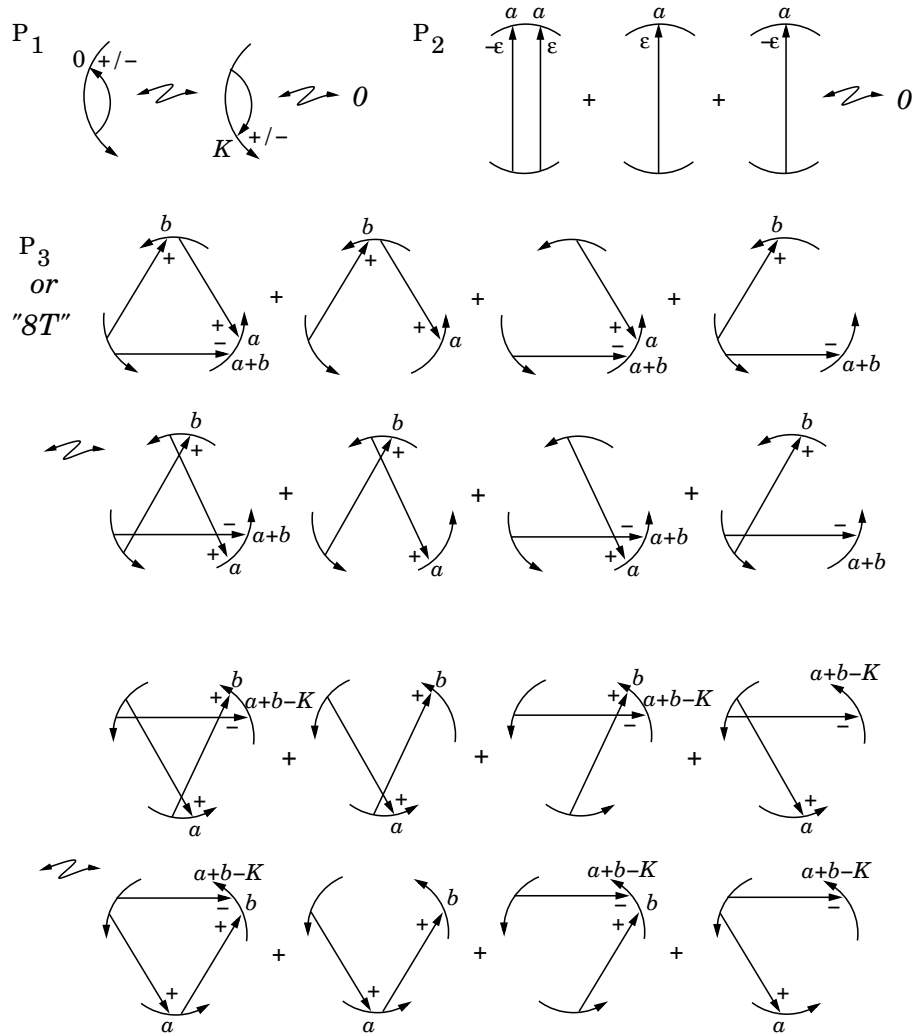,scale=0.77}
\caption{Relations defining the Polyak algebra in the solid torus framework}\label{pic:Pmoves}
\end{figure}

 The following theorem repeats Theorem $2.D$ from \citep{GPV} -- the proof is similar.

\begin{theorem}\label{thm:PolyakAlg}
The map $I$ induces an isomorphism $\mathcal{G}_{\text{\!\reflectbox{$\setminus$}}\mathrm{R}}\rightarrow \mathcal{G}_{\text{\!\reflectbox{$\setminus$}}\mathrm{P}}\stackrel{def}{=}\mathcal{P}$, where $\mathrm{R}$ stands for the Reidemeister relations on Gauss diagrams. More precisely, $I$ induces an isomorphism between $\operatorname{Span}(\mathrm{R}_i)$ and $\operatorname{Span}(\mathrm{P}_i)$, for $i=1,2,3$.
\end{theorem}

\subsection{Based and degenerate diagrams}

A \textit{based} Gauss diagram is a Gauss diagram together with a distinguished (\textit{base}) arc on the circle, \ie a region between two consecutive ends of arrows. Based arrow diagrams are defined similarly. The corresponding spaces are denoted by $\mathcal{G}_\bullet$ and $\mathcal{A}_\bullet$, in reference to the dot that we use in practice to pinpoint the distinguished arc.

A \textit{degenerate Gauss diagram (with one degeneracy)} is a classical Gauss diagram in which one of the $2n$ arcs of the base circle has been shrunk to a point. When the arc was bounded by the two endpoints of one and the same arrow, the degenerate diagram is decorated with the datum of which endpoint was before the other. In this way, there is a natural $1$-$1$ correspondence between based and degenerate diagrams. The spaces of degenerate diagrams are called $\mathcal{DG}$ and $\mathcal{DA}$ respectively. The latter is meant to be quotiented by the so-called \textit{triangle relations}, shown in Fig.\ref{pic:triangle}.
The quotient space is denoted by $\mathcal{DA}_{/\nabla}$.

\begin{figure}[h!]
\centering 
\psfig{file=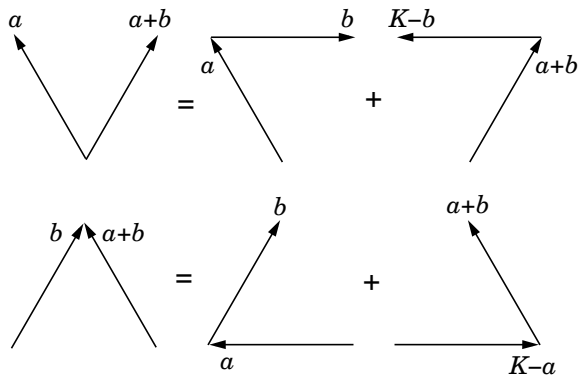,scale=0.64}
\caption{The triangle relations}\label{pic:triangle}
\end{figure}

This notion in the context of $\SS^3$ is due to M. Polyak, and is part of his conjecture which we discuss in the next section.

\subsection{Polyak's conjecture}
During Swiss Knots conference in 2011, Michael Polyak gave a talk in which he conjectured that Gauss diagram formulas for knots in $\SS^3$ were the space of solutions to an equation of type $d(\cdot)=0$ with $d$ valued in $\mathcal{DA}_{/\nabla}$. We shall not give a formal statement of this here, since it was never written by its author, but there is a video of the talk available online (\citep{PolyakTalk}). The map $d$ from that conjecture has the property of being homogeneous:
$$\forall G\in \mathcal{G}, \left[ d(G)=0 \Longleftrightarrow \forall n\in \NN, d\circ \pi_n (G)=0 \right].$$
It follows from \ref{sec:homogeneous} that, in the virtual setting, the best that we may expect from such a map is to detect \textit{arrow} diagram formulas. For this reason, from now on we mostly restrict our attention to this kind of invariants. We construct a map $d$ in that framework, that differs from M.Polyak's one by the signs in front of the contributing diagrams. Then we prove that its kernel encodes Reidemeister III invariance, while Reidemeister I and II are already very easy to check.

\begin{remark}Based on our understanding of the conjecture, Polyak-Viro's formula for $v_3$ (Fig.\ref{pic:v3}) is a counterexample in the \enquote{classical} settings: it defines an invariant of usual knots, but it does not have a trivial boundary, no matter how the signs are chosen to compute it. Hence the present result seems to be the best one can hope for.
\end{remark}

\section{Main results}

\subsection{A set of equations for virtual arrow diagram formulas}

In this section we define a map $d$ that will fit in some version of Polyak's conjecture for virtual knots in the solid torus.

\subsubsection*{Homogenous Polyak relations}
Let $G\in\mathcal{G}_n$. Then $G$ satisfies the $\mathrm{P}_i$ relations (\ie $\left\langle G,\mathrm{P}_i\right\rangle =0)$) if and only if it satisfies the \textit{homogeneous} relations $\left\langle G,\pi_k(\mathrm{P}_i)\right\rangle =0)$ for all $k$. These are denoted by $\mathrm{P}_1$, $\mathrm{P}_2^{(n-1),1}$, $\mathrm{P}_2^{(n-2),2}$, $\mathrm{P}_3^{(n-2),2}$ (or $G6T$) and $\mathrm{P}_3^{(n-3),3}$ (or $G2T$) -- some examples are shown on Fig.\ref{pic:homogeneous}. The parenthesized numbers indicate in each case how many arrows are unseen.

\begin{figure}[h!]
\centering 
\psfig{file=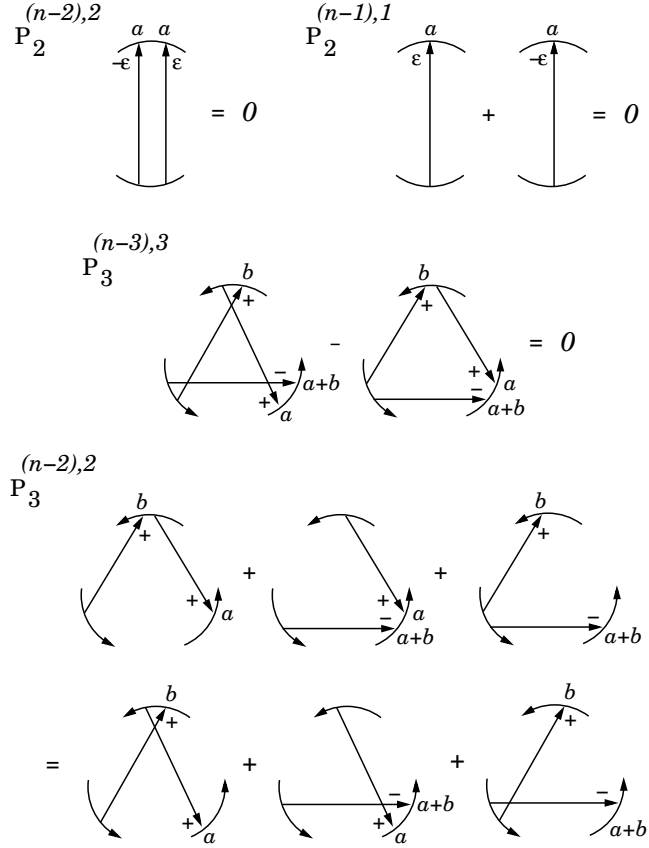,scale=0.6}
\caption{Some homogeneous Polyak relations}\label{pic:homogeneous}
\end{figure}

\begin{lemma}\label{ImS}
Let $G\in\mathcal{G}$. Then $G$ lies in the image of the map $S$ if and only if $G$ satisfies all the homogeneous relations $\left\langle G,\mathrm{P}_2^{(n-1),1}\right\rangle =0$.
\end{lemma}

Homogenous relations are also defined for arrow diagram spaces. This time one should pay attention to signs, so we make a full list (Fig.\ref{pic:Arelations}). We denote them by $\mathrm{AP}_1$, $\mathrm{AP}_2^{(n-2),2}$, $\mathrm{AP}_3^{(n-2),2}$ (or $A6T$) and $\mathrm{AP}_3^{(n-3),3}$ (or $A2T$). The above lemma explains why $\mathrm{AP}_2^{(n-1),1}$ is useless (it writes $0=0$).

Let us mention here the following two crucial points in the proofs of Theorems~\ref{thm:ArrowHomogenous} and ~\ref{thm:main}.

\begin{lemma}\label{crucial2T6T} For all $n\geq 3$:
$$\operatorname{Span}(\mathrm{AP}_3^{(n-3),3}) \subseteq \operatorname{Span}(\mathrm{AP}_3^{(n-2),2})\cup \operatorname{Span}(\mathrm{AP}_2^{(n-2),2}).$$
\end{lemma}

\begin{lemma}\label{SpanOrth}
Let $A\in \mathcal{A}$ and let $\mathrm{X}$ be a name among $\mathrm{P}_1$, $\mathrm{P}_2^{(n-2),2}$, $\mathrm{P}_3^{(n-2),2}$, $\mathrm{P}_3^{(n-3),3}$. Then

$$A\in \operatorname{Span}^\perp(\mathrm{AX})\Longleftrightarrow S(A)\in \operatorname{Span}^\perp(\mathrm{X}).$$
\end{lemma}

\begin{figure}[h!]
\centering 
\psfig{file=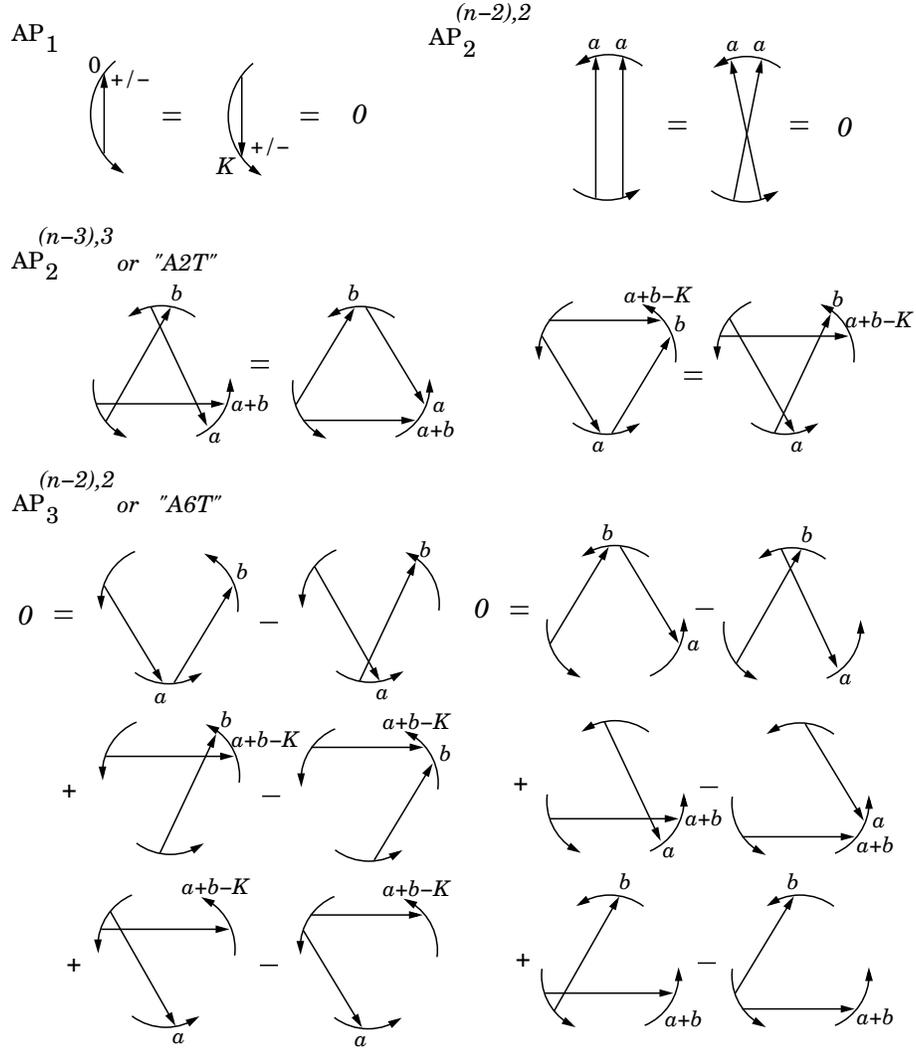,scale=0.60}
\caption{The homogeneous arrow relations}\label{pic:Arelations}
\end{figure}

\subsubsection*{Defining the map $d$}

Applying Theorem~\ref{thm:PolyakAlg}, Lemma~\ref{ImS} and Lemma~\ref{SpanOrth}, gives immediately:

\begin{lemma}\label{thm:R1et2}
Let $A\in \mathcal{A}$. Then the function $ \left\langle S(A),I(\cdot) \right\rangle $ defines an invariant under Reidemeister I and II moves for virtual knots with homology class $K$ if and only if $A$ satisfies all the relations $\left\langle A,\mathrm{AP}_1\right\rangle =0$ and $\left\langle A,\mathrm{AP}_2^{(n-2),2}\right\rangle =0$.
\end{lemma}

This condition is easy to check with our naked eye, so we will be happy with a map $d$ which can only detect invariance under Reidemeister III.
\newline

Let $A$ be an arrow diagram. We denote by $\bullet(A)\in \mathcal{A}_\bullet$ the sum of all based diagrams that one can form by choosing a base arc in $A$. The map $d$ is first going to be defined on based diagrams.

\begin{definition}
We say that a based diagram $B_\bullet$ is \textit{nice} if the endpoints of its base arc belong to two different arrows.
\end{definition}
If $B_\bullet$ is not nice, then we set $d(B_\bullet)=0$.

If $B_\bullet$ is nice, then $d(B_\bullet)$ is the degenerate diagram obtained from $B_\bullet$ by shrinking the base to a point, multiplied by a sign $\epsilon (B_\bullet)$ defined as follows. Put
$$\eta (B_\bullet)=\left\lbrace \begin{array}{l}
+1 \text{ if the arrows that bound the base arc cross each other} \\
-1 \text{ otherwise}
\end{array}\right. ,$$
and let $\uparrow\!(B_\bullet)$ be the number of arrowheads at the boundary of the base arc.
Then
$$\epsilon (B_\bullet)= \eta(B_\bullet)\cdot(-1)^{\uparrow(B_\bullet)}.$$
The map $d$ is extended linearly to $\mathcal{A_\bullet}$.

Finally, define:
\begin{eqnarray}
d(A)=d(\bullet(A))\in \mathcal{DA}_{\text{\!\reflectbox{$\setminus$}}\nabla}.
\end{eqnarray}

An example is shown on Fig.\ref{PCI2}.

\begin{figure}[h!]
\centering 
\psfig{file=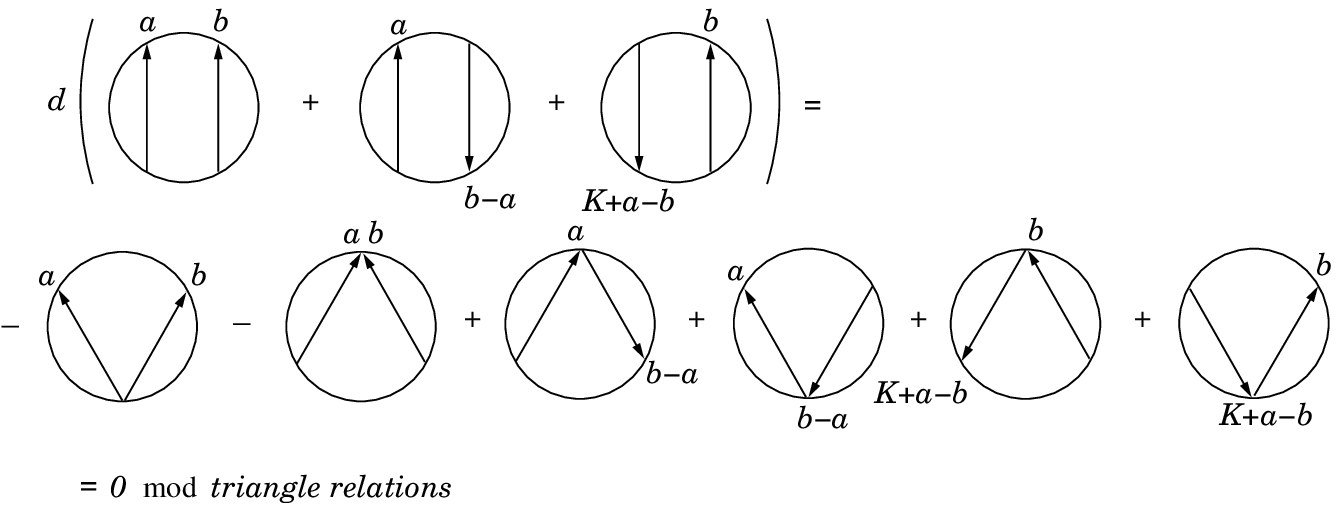,scale=0.9}
\caption{Boundary of the universal planar chain invariant for $n=2$}\label{PCI2}
\end{figure}

\begin{figure}[h!]
\centering 
\psfig{file=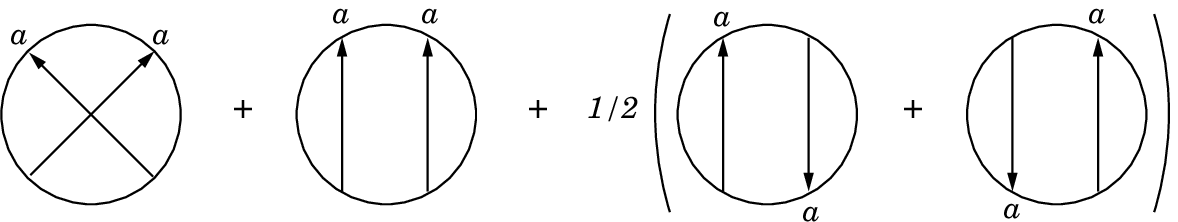
,scale=0.9}
\caption{An arrow polynomial in the kernel of $d$ but not $\mathrm{R}_2$-invariant}\label{pic:counterexample}
\end{figure}

\begin{theorem}[Main Theorem]\label{thm:main}
Let $A\in \mathcal{A}$ satisfy the conditions of Lemma~\ref{thm:R1et2}. Then the following are equivalent:
\begin{enumerate}
\item $A$ is an arrow diagram formula for invariants of virtual knots.
\item $d(A)=0$ modulo the triangle relations.
\item $A\in \operatorname{Span}^\perp (A6T)$.
\end{enumerate}
\end{theorem}

This theorem gives a formal proof to the fact that the $ \left\langle , \right\rangle $ pairing enables a uniformization of the formulas that depend on parameters: we shall see for instance that the degenerate cases of Grishanov-Vassiliev's invariants need not be treated separately. See also Proposition $2$ from \citep{Fiedlerbraids}, where points $(i)$ and $(ii)$ are actually the same formula, or the special case $2a=K$ in Theorem~\ref{thm:length5} below. Note that the requirement of $\mathrm{R}_2$ invariance is necessary with the present settings -- see Fig.\ref{pic:counterexample}.

\begin{remark}
It is possible to give a simplicial interpretation of the above map $d$, that makes it the first step towards a cohomology theory \enquote{\`{a} la Vassiliev} \citep{Vassiliev1990,Vassiliev}. It will be done in a forthcoming paper.
\end{remark}



\subsection{Application to Grishanov-Vassiliev's planar chain invariants}

In \citep{Grishanov}, Grishanov-Vassiliev define an infinite family of arrow diagram formulas for classical knots in $M^2\times \RR$. Let us recall their construction.
\begin{definition}\label{def:nakedArrow}
A \textit{naked arrow diagram} is an arrow diagram with every decoration forgotten -- as usual, up to oriented homeomorphism of the circle.

A naked arrow diagram is called \textit{planar} if no two of its arrows intersect.

A \textit{chain presentation} of such a diagram with $n$ arrows is a way to number its $n+1$ internal regions from $1$ to $n+1$, in such a way that the numbering increases when one goes from the left to the right of an arrow.

Let $U_n$ be the sum of all planar isotopy equivalence classes of chain presentations of naked arrow diagrams with $n$ arrows. $U_n$ is called the \textit{universal degree $n$ planar chain} (\citep{Grishanov}, Definition 1).
\end{definition}

For $i=1,\ldots,n+1$, let $\gamma_i\in \ZZ\setminus\left\lbrace 0\right\rbrace$. Given a chain presentation of a naked planar arrow diagram, and given such an ordered collection $\Gamma$, we construct an arrow diagram by decorating each arrow with the sum of the $\gamma_i$'s whose index $i$ is located to the left of that arrow. The global decoration of the circle is set to be the sum of all $\gamma_i$'s.

The element of $\mathcal{A}_n$ constructed that way from $U_n$ and $\Gamma$ is denoted by $\Phi_\Gamma$ (\citep{Grishanov}, Definition 2). Note that some of the summands in $U_n$ may lead to the same element of $\mathcal{A}_n$ if some of the $\gamma_i$'s are equal; unlike Vassiliev-Grishanov, we do not forbid that.

The arrow polynomial at the top of Fig.\ref{PCI2} is the generic example for $n=2$, with $\Gamma=(a,b-a,K-b)$.

\begin{theorem}\label{thm:Grishanov}
For any ordered collection of non-zero homology classes $\Gamma=(\gamma_1,\ldots,\gamma_{n+1})$, the sum $\Phi_\Gamma$ defined above enjoys the hypotheses of Theorem~\ref{thm:main} and thus defines an arrow diagram formula for virtual knots.
\end{theorem}

This is an improvement of Theorem $1$ from \citep{Grishanov}, since we remove the assumption that $\Gamma$ is unambiguous (\ie here any of the $\gamma_i$'s may coincide), and we show that $\Phi_\Gamma$ is an invariant for virtual knots. 

\subsection{Some more computations}
In practice, Theorem~\ref{thm:main} gives a very easy means of checking that an arrow polynomial defines a virtual invariant. On another hand, finding virtual invariants when one has no clue of a potential formula demands to solve the system of equations $A6T$.

We wrote a program to do this, and only a few results came, including the generalized Grishanov-Vassiliev's planar chain invariants, and the following:

\begin{theorem}\label{thm:length5}
Let $K\in \ZZ$.
The arrow polynomial of Fig.\ref{pic:length5} defines an invariant of virtual knots with homology class $K$ for any parameter $a\in \ZZ\setminus \left\lbrace 0\right\rbrace$.
\end{theorem}

This seems to give a positive answer to T. Fiedler's question about the existence of $N$-invariants not contained in those from \citep{Fiedlerbraids}, Proposition $2$.

On the other hand, the sparse landscape of results leads to think that most arrow diagram invariants might have infinite length -- \ie live in the algebraic completion of $\mathcal{A}$, just like Fiedler's $N$-invariants.

\begin{figure}[h!]
\centering 
\psfig{file=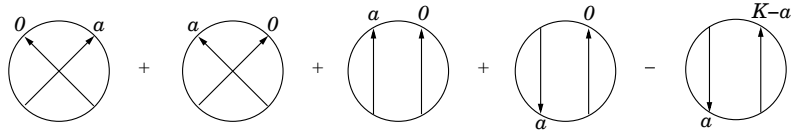,scale=1.1}
\caption{An invariant of length $5$ identically zero for closed braids}\label{pic:length5}
\end{figure}

Notice that in case $a=K$, the formula has only $3$ terms -- but still defines an invariant. We compute it for the family of knots $\mathcal{K}_{2i+1}$ drawn in Fig.\ref{K2i}:
$$I_5(\mathcal{K}_{2i+1})=i(i+1).$$
It proves that there is no general algebraic formula expressing $I_5$ in terms of the only invariants of degree $2$ and finite length previously known (at least to the author), namely Grishanov-Vassiliev's length $3$ invariant (Fig.\ref{PCI2}) -- note that this invariant is already present in \citep{Fiedler} for nullhomologous knots ($K=0$).

\begin{figure}[h!]
\centering 
\psfig{file=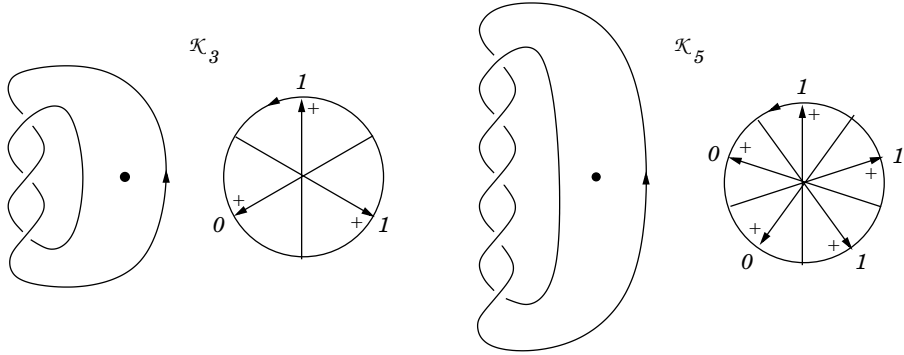,scale=0.9}
\caption{$\mathcal{K}_3$ and $\mathcal{K}_5$}\label{K2i}
\end{figure}

\section{Proofs}

Recall the notations from section~\ref{sec:arrows}.
\begin{lemma}\label{lem:brackets}
For all $A\in \mathcal{A}$ and $G\in \mathcal{G}$, the following equality holds: $$ \left\langle \! \left\langle A,G \right\rangle \! \right\rangle = \left\langle S(A),I(G) \right\rangle .$$
The equality $$(\!(A,G)\!) = (S(A),I(G))$$ holds for all $G$ if and only if $A$ has no symmetries other than the identity.
\end{lemma}

\begin{proof}
Number the $n$ arrows of $A$ and fix a map $\sigma_0:\left\lbrace 1,\ldots,n \right\rbrace\rightarrow\left\lbrace \pm 1\right\rbrace$.
The group of symmetries of $A^{\sigma_0}$, $\operatorname{Aut}(A^{\sigma_0})$, identifies with a subgroup of $\operatorname{Aut}(A)$. Also, the (abelian) group $\operatorname{Aut}(A)$ naturally acts on the set $\left\lbrace \pm 1\right\rbrace^n$. The orbit of $\sigma_0$ is the set of maps $\sigma$ such that $A^\sigma$ is equivalent to $A^{\sigma_0}$ under planar isotopies, and the cardinality of this orbit is the coefficient of $A^\sigma$ in the linear combination $S(A)$. By definition, the stabilizer of $\sigma_0$ is the image of the injective map $\operatorname{Aut}(A^{\sigma_0})\hookrightarrow \operatorname{Aut}(A)$, whence, if we set $\mathcal{O}$ to be the set of orbits:
$$S(A) =\sum_{[\sigma]\in \mathcal{O}} \frac{\abs{\operatorname{Aut}(A)}}{\abs{\operatorname{Aut}(A^\sigma)}}A^\sigma.$$
It follows that for any $G$,
$$
\begin{array}{ccc}
\vspace{2mm} \left\langle S(A),I(G) \right\rangle  & = & \abs{\operatorname{Aut}(A)}\sum_{[\sigma]\in \mathcal{O}}\operatorname{sign}(\sigma)(A^\sigma,I(G)) \\ \vspace{2mm}
& = & \abs{\operatorname{Aut}(A)}(\!(A,G)\!) \\
& = &  \left\langle \! \left\langle A,G \right\rangle \! \right\rangle . \end{array}$$
This proves the first equality as well as the \enquote{if} part of the last statement. For the \enquote{only if} part, set $G$ to be $A^{\sigma_0}$. Then:
$$\begin{array}{ccc}
\vspace{2mm}
(S(A),I(G)) & = & \sum_{[\sigma]\in \mathcal{O}}\operatorname{sign}(\sigma)\frac{\abs{\operatorname{Aut}(A)}}{\abs{\operatorname{Aut}(A^\sigma)}}(A^\sigma,I(G)) \\
& = & \operatorname{sign}(\sigma_0)\frac{\abs{\operatorname{Aut}(A)}}{\abs{\operatorname{Aut}(A^{\sigma_0})}},
\end{array}$$
while
$$(\!(A,G)\!)= \operatorname{sign}(\sigma_0).$$
So one must have $\abs{\operatorname{Aut}(A)}=\abs{\operatorname{Aut}(A^\sigma)}$ for all $\sigma$, which can be true only if $\abs{\operatorname{Aut}(A)}=1$. Indeed, if $\rho\in \operatorname{Aut}(A)\setminus \left\lbrace \mathrm{Id}\right\rbrace $, pick an arrow $\alpha$ of $A$ such that $\rho(\alpha)\neq \alpha$ and choose any $\sigma$ such that $\sigma(\alpha)=1$ while $\sigma(\rho(\alpha))=-1$. Necessarily
$\rho\notin \operatorname{Aut}(A^\sigma)$, so that  $\abs{\operatorname{Aut}(A^\sigma)} < \abs{\operatorname{Aut}(A)}$. 

\end{proof}

\begin{proof}[Proof of Lemma~\ref{crucial2T6T}.]

Figs.\ref{pic:eight6T} and \ref{pic:eight6Tter} show eight $A6T$ relations, where it is assumed that the unseen parts are identical in all 48 diagrams (the dashed arrows are dashed only for the sake of clarity). Up to $\mathrm{P}_2^{(n-2),2}$, the combination $$(1)+(2)+\frac{1}{2}\left[ (3)+(4)-(5)-(6)-(7)-(8)\right]$$ gives the top $A2T$ relation shown on Fig.\ref{pic:Arelations}. To get the other half of $\operatorname{Span}(A2T)$, just reverse the arrows in the previous equation, and change their markings from $x$ to $K-x$.

\end{proof}

\begin{figure}[h!]
\centering 
\psfig{file=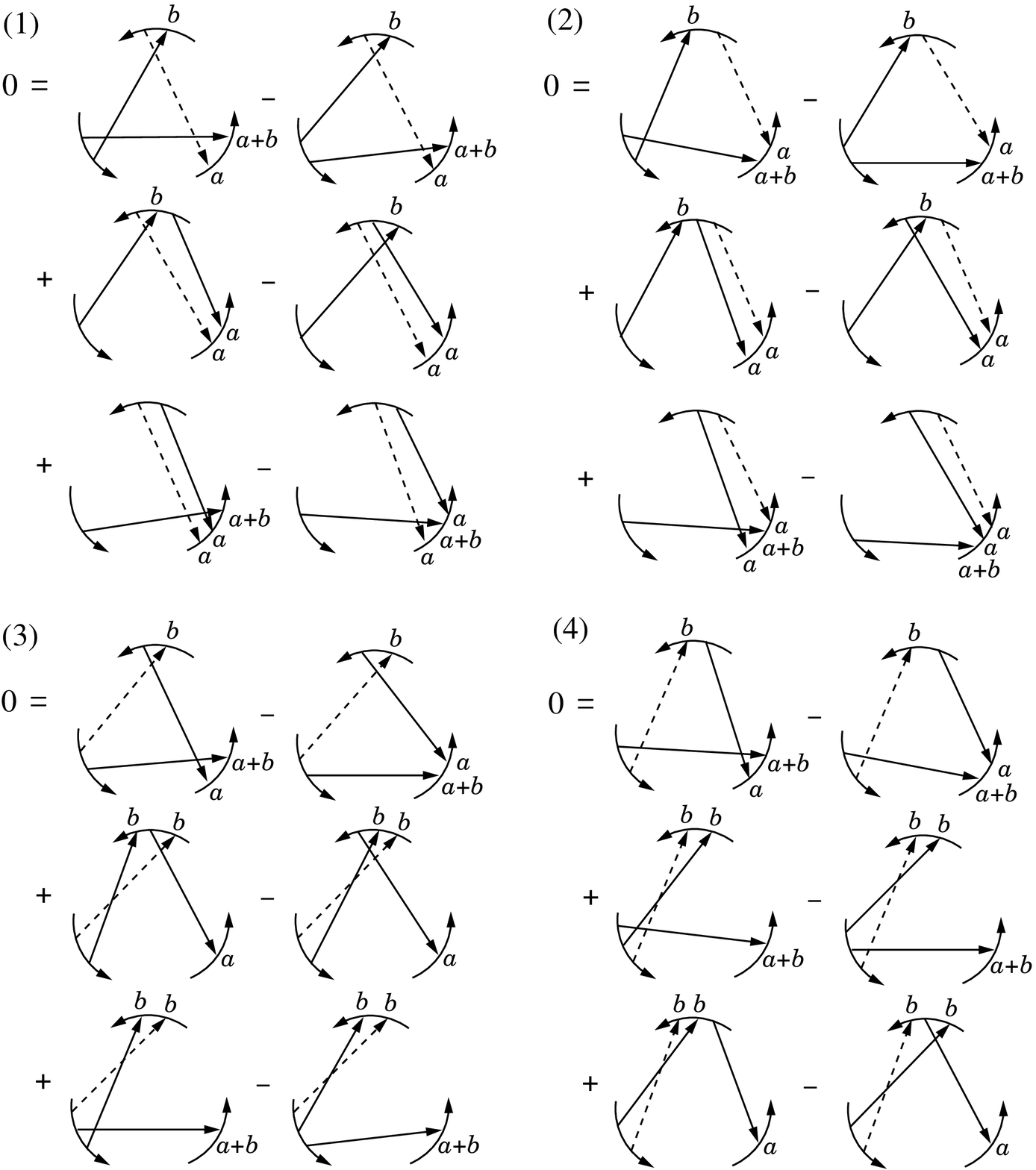,scale=0.60}
\caption{Proof of Lemma~\ref{crucial2T6T} -- part $1$}\label{pic:eight6T}
\end{figure}

\begin{figure}[h!]
\centering 
\psfig{file=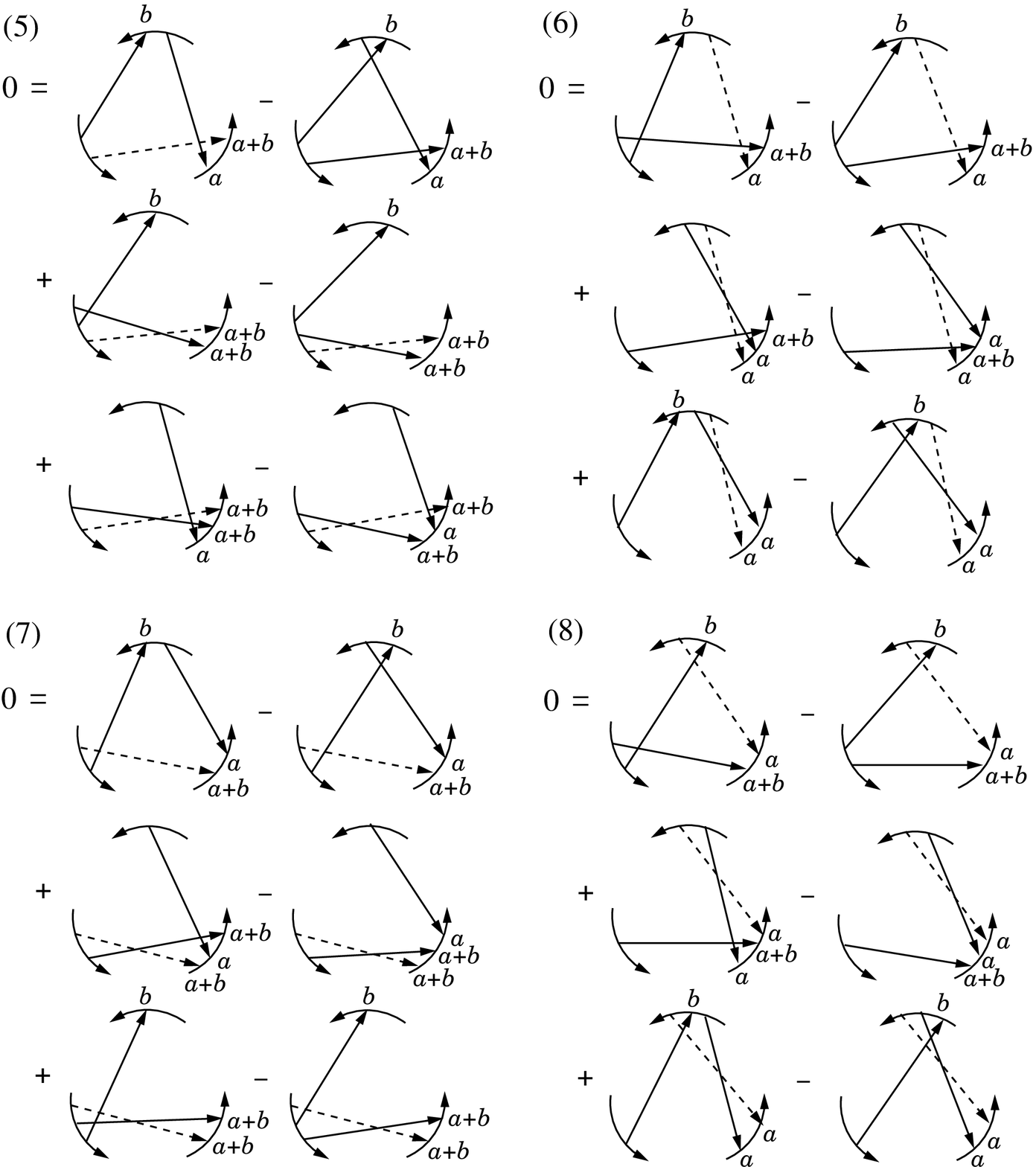,scale=0.60}
\caption{Proof of Lemma~\ref{crucial2T6T} -- part $2$}\label{pic:eight6Tter}
\end{figure}

\begin{proof}[Proof of Lemma~\ref{PrincipalPart}.]
Let $G_d$ be the principal part of $G$. Since $G$ is a Gauss diagram formula it must satisfy the $\mathrm{P}_2$ relations (Theorem~\ref{thm:PolyakAlg}). Since $G$ does not have summands of degree higher than $d$, $G_d$ must satisfy the $P_2^{(n-1),1}$ relations. Lemma~\ref{ImS} concludes the proof.
\end{proof}

\begin{proof}[Proof of Lemma~\ref{ImS}.]

Let $A$ be an arrow diagram, and set $G=S(A)$. By definition of $S$, any couple of Gauss diagrams that differ only by the sign of one arrow happen in $G$ with opposite coefficients. $S$ being linear, this implies:
$$S(\mathcal{A})\subset \bigoplus_{n\geq 1} \operatorname{Span}^\perp(\mathrm{P}_2^{(n-1),1}).$$
On the other hand, let $G$ satisfy the $\mathrm{P}_2^{(n-1),1}$ equations. Define 
$$A=\sum (G,\mathtt{A}^+)\cdot {\mathtt{A}},$$
where the sum runs over all arrow diagrams, and the $+$ operator decorates every arrow with a $+$ sign.
It is easy to check that $G=S(A)$.

\end{proof}

\begin{proof}[Proof of Theorem~\ref{GaussComplete}.]
By Theorem $2.4$ from \citep{MortierGaussDiagrams}, a virtual knot diagram may be recovered from its Gauss diagram. It remains to see the correspondence between classical Reidemeister moves and $\mathrm{R}$-moves.
It relies on the easy fact that an $\mathrm{R}$-move truly corresponds to a Reidemeister move picture if and only if the little loops shown in Fig.\ref{pic:loops} are nullhomologous. 
We use Theorem $1.1$ from \citep{P2} to conclude: the $\mathrm{R}_3$ moves from Fig.\ref{pic:Rmoves} are the only Gauss pictures that can match $\Omega 3a$. Since M. Polyak's proof is local, it works in our framework as well. 

\end{proof}

\begin{figure}[h!]
\centering 
\psfig{file=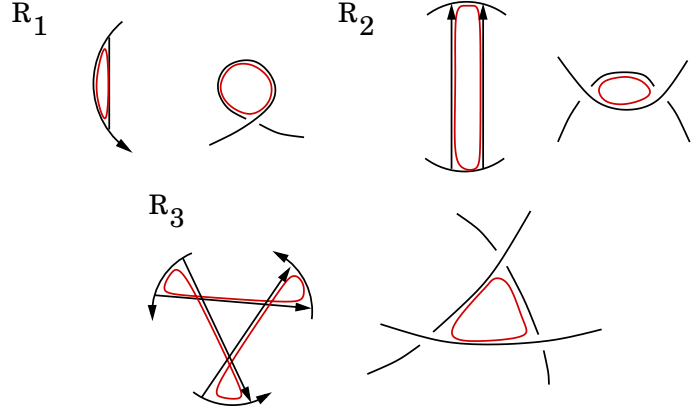,scale=0.70}
\caption{Homological obstruction to $\mathrm{R}$-moves}\label{pic:loops}
\end{figure}

\begin{proof}[Proof of Theorem~\ref{thm:ArrowHomogenous}.]
Let $A\in \mathcal{A}$ be an arrow diagram formula. It suffices to prove that the principal part of $A$, say $A_d$, is an arrow diagram formula. By Theorem~\ref{thm:PolyakAlg}, $\left\langle S(A),\mathrm{P}_i\right\rangle =0$ for $i=1,2,3$. Let us show that the same goes for $A_d$:
\newline

$1.$ $\mathrm{P}_1$ and $S$ are homogeneous, so $\left\langle S(A_d),\mathrm{P}_1\right\rangle =0$.
\newline

$2.$ By Lemma~\ref{ImS}: 
\begin{eqnarray}\label{eq5}
\left\langle S(A_{d-1}),P_2^{(d-2),1}\right\rangle  & = & 0\\\label{eq6}
\left\langle S(A_{d}),P_2^{(d-1),1}\right\rangle  & = & 0.
\end{eqnarray}
The equations $\left\langle S(A),\mathrm{P}_2\right\rangle =0$ together with~\ref{eq5} imply that 
\begin{eqnarray}\label{eq7}
\left\langle S(A_d),\mathrm{P}_2^{(d-2),2}\right\rangle =0.
\end{eqnarray}
Together with~\ref{eq6}, we get:
$\left\langle S(A_d),\mathrm{P}_2\right\rangle =0$.
\newline

$3.$ The last and crucial point:
 $$\begin{array}{cc}
\hspace{-1.5cm}\left\langle S(A),\mathrm{P}_3\right\rangle =0 \,\,\, \Longrightarrow & \left\langle S(A_d),\mathrm{P}_3^{(d-2),2}\right\rangle =0 \\
\hspace{1cm} \stackrel{\ref{eq7}\text{ + Lemma}~\ref{SpanOrth}}{\Longrightarrow} & \left\langle A_d,\mathrm{AP}_3^{(d-2),2}\right\rangle =0  \text{ and } \left\langle A_d,\mathrm{AP}_2^{(d-2),2}\right\rangle =0\\
\hspace{1cm} \stackrel{\text{Lemma}~\ref{crucial2T6T}}{\Longrightarrow} & \left\langle A_d,\mathrm{AP}_3^{(d-2),2}\right\rangle =0  \text{ and }\left\langle A_d,\mathrm{AP}_3^{(d-3),3}\right\rangle =0 \\
\hspace{1cm} \stackrel{\text{Lemma}~\ref{SpanOrth}}{\Longrightarrow} & \left\langle S(A_d),\mathrm{P}_3^{(d-2),2}\right\rangle 0 \text{ and }\left\langle S(A_d),\mathrm{P}_3^{(d-3),3}\right\rangle =0 \\
\hspace{1cm} \Longrightarrow & \left\langle S(A_d),\mathrm{P}_3\right\rangle =0.

\end{array}$$ 
\end{proof}

\begin{proof}[Proof of Theorem~\ref{thm:main}.]
The proof will consist in defining and explaining the following chain of equivalences. 
$$d(A)=0\Leftrightarrow (\bullet(A),A6T_\bullet)=0\Leftrightarrow\,\,  \left\langle A,A6T \right\rangle =0\Leftrightarrow\,\,  \left\langle S(A),I(\mathrm{R}_3) \right\rangle =0$$

Notice that both extremities of this chain are \textit{homogeneous} conditions (for the right one, it follows from the proof of Theorem~\ref{thm:ArrowHomogenous}). So we may assume that $A$ is homogeneous.

$1.$ $d(A)=0\Leftrightarrow (\bullet(A),A6T_\bullet)=0$.\newline
Let us call a degenerate diagram (with one degeneracy) \textit{monotonic} if an arrowhead and an arrowtail meet at the degenerate point. The set of monotonic diagrams forms a basis of $\mathcal{DA}_{/\nabla}$. It is clearly a generating set thanks to the $\nabla$ relations, and it is free because every non monotonic diagram happens in exactly one relation, and every relation contains exactly one of them.

We introduce the orthonormal scalar product $(,)$ with respect to this basis.

Let $D$ be a monotonic degenerate diagram and $B_\bullet$ a based diagram. It is easy to check that the coordinate of $d(B_\bullet)$ along $D$ is given by $$(d(B_\bullet),D)=(B_\bullet,A6T_\bullet(D)),$$
where $A6T_\bullet (D)$ is what we call the \textit{based} $6$-term relation associated with $D$ (see an example on Fig.\ref{pic:based6T}).

\begin{figure}[h!]
\centering 
\psfig{file=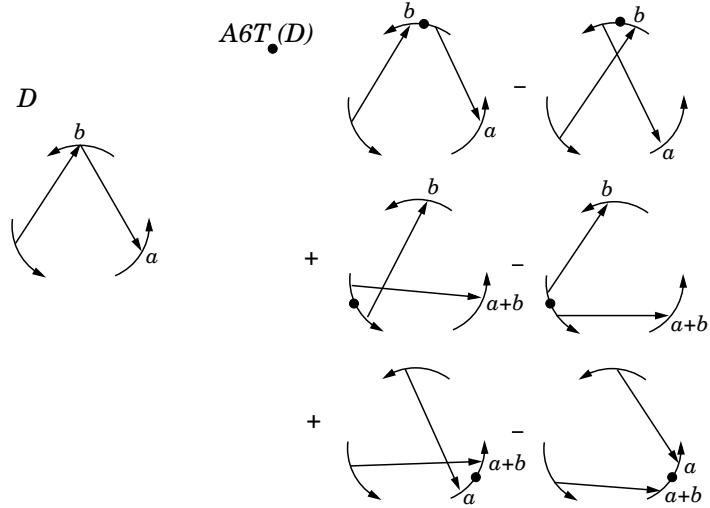,scale=0.60}
\caption{The based 6-term relation associated with a degenerate diagram}\label{pic:based6T}
\end{figure}

$2.$ $(\bullet(A),A6T_\bullet)=0\Leftrightarrow \,\, \left\langle A,A6T \right\rangle =0$. \newline
Let $A$ and $A^\prime$ denote two arrow diagrams. We set:
$$[A,A^\prime]\stackrel{def}{=}(\bullet(A),A^\prime_\bullet)$$
where $A^\prime_\bullet$ is the based diagram obtained from $A^\prime$ by choosing \textit{any} arc as a base arc. We have to show that this is a definition.
If $A\neq A^\prime$, then the right hand side is unambiguously zero.
If $A=A^\prime$, then there are exactly $\abs{\operatorname{Aut}(A^\prime)}$ summands in $\bullet(A)$ that coincide with any fixed choice of base point in $A^\prime$. So the pairing $[,]$ is well-defined, and moreover it coincides with $ \left\langle , \right\rangle $.\newline

$3.$ $ \left\langle A,A6T \right\rangle =0\Leftrightarrow\,\,  \left\langle S(A),I(\mathrm{R}_3) \right\rangle =0$

By Theorem~\ref{thm:PolyakAlg}, $ \left\langle S(A),I(.) \right\rangle $ being invariant under $\mathrm{R}_3$ moves is equivalent to $ \left\langle S(A),\mathrm{P}_3 \right\rangle =0$ for any Polyak's $8T$ relation $\mathrm{P}_3$. Since by hypothesis $A$ is homogeneous (say of degree $n$), this is equivalent to $S(A)$ actually satisfying separately the $\mathrm{P}_3^{(n-3),3}$ and the $\mathrm{P}_3^{(n-2),2}$ relations. Now apply successively Lemmas~\ref{SpanOrth} and~\ref{crucial2T6T} to terminate the proof.

\end{proof}
\begin{proof}[Proof of Theorem~\ref{thm:Grishanov}.]
The fact that no $\gamma_i$ may be trivial gives immediately the condition from \ref{thm:R1et2}.
It is convenient here to check condition $3$ of Theorem~\ref{thm:main}. In any $A6T$ relation, only three diagrams may have pairwise non intersecting arrows, and either all of these have, either no one has. The subsequent reduced relations are shown on Fig.\ref{pic:reduced6T} (the usual relations between the markings of the arrows have a natural equivalent in terms of Grishanov-Vassiliev's region markings). We say that a diagram with its regions marked is \textit{consistent} if its markings satisfy the \textit{chain presentation} rule from Definition~\ref{def:nakedArrow} -- in other words, a diagram is consistent if it appears in $\Phi_\Gamma$. Consider the top relation of Fig.\ref{pic:reduced6T}, which can be written $A_1-A_2-A_3$. We see that:
\begin{enumerate}
\item $A_2$ is consistent if and only if $A_1$ is consistent and $i < j$.
\item $A_3$ is consistent if and only if $A_1$ is consistent and $i > j$.
\end{enumerate}
It follows that $\Phi_\Gamma$ satisfies the $A6T_1$ relations. The proof for $A6T_2$ is similar.
 
\begin{figure}[h!]
\centering 
\psfig{file=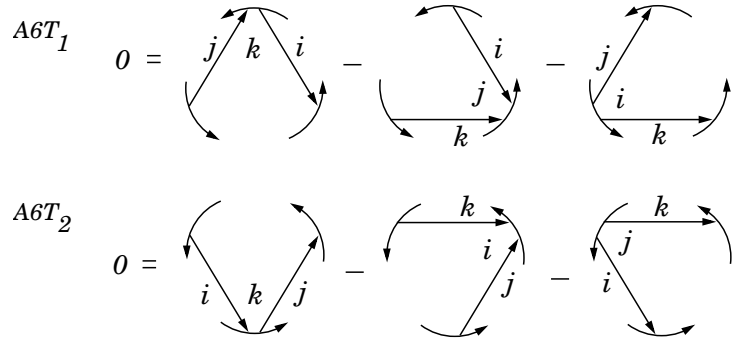,scale=0.60}
\caption{The two kinds of reduced 6-term relations for planar diagrams}\label{pic:reduced6T}
\end{figure}

\end{proof}

\bibliographystyle{plain}
\bibliography{bibli}

\vfill
Institut de Mathematiques de Toulouse

Universite Paul Sabatier et CNRS (UMR 5219)

118, route de Narbonne

31062 Toulouse Cedex 09, France

mortier@math.ups-tlse.fr

\end{document}